\documentclass[12pt, reqno]{amsart}
\usepackage{amsmath, amsthm, amscd, amsfonts, amssymb, graphicx, color, hyperref}
\newtheorem{theorem}{Theorem}[section]
\newtheorem{lemma}[theorem]{Lemma}
\newtheorem{proposition}[theorem]{Proposition}
\newtheorem{corollary}[theorem]{Corollary}
\theoremstyle{definition}

\newtheorem{example}[theorem]{Example}

\theoremstyle{remark}
\newtheorem{remark}[theorem]{Remark}
\numberwithin{equation}{section}

\begin{document}
\setcounter{page}{1}

\title[Harnack parts]{Harnack parts of $\displaystyle \rho$-contractions}
\author[G. Cassier, M. Benharrat, S. Belmouhoub]{Gilles Cassier$^{1*}$, Mohammed Benharrat$^2$ and Soumia Belmouhoub$^3 $}
\address{$^{1}$ Universit\'e de Lyon 1; Institut Camille Jordan CNRS UMR 5208; 43, boulevard du 11 Novembre 1918, F-69622 Villeurbanne.
}

\email{\textcolor[rgb]{0.00,0.00,0.84}{cassier@math.univ-lyon1.fr}}
\address{$^{2}$ D\'epartement de Math\'ematiques et informatique, Ecole Nationale Polytechnique d'Oran (Ex. ENSET d'Oran);  B.P. 1523  Oran-El M'Naouar, Oran, Alg\'erie.
}
\email{\textcolor[rgb]{0.00,0.00,0.84}{mohammed.benharrat@enp-oran.dz}}
\address{$^{3}$ D\'epartement de Math\'ematiques, Universit\'e de Mostaganem; Alg\'erie.
}
\email{\textcolor[rgb]{0.00,0.00,0.84}{belmsou@yahoo.fr}}




\date{14/12/2016.
\newline \indent $^{*}$ Corresponding author}
\begin{abstract}
The purpose of this paper is to describe the Harnack parts for the operators of class $C_{\rho}$ ($\rho>0$) on Hilbert spaces which were introduced by B. Sz. Nagy and C. Foias in \cite{SzNF}. More precisely, we study Harnack parts of operators with $\rho$-numerical radius one. The case of operators with $\rho$-numerical radius strictly less than $1$ was described in \cite{CaSu}. We obtain a general criterion for compact $\rho$-contractions to be in the same Harnack part. We give a useful equivalent form of this criterion for usual contractions. Operators with numerical radius one received also a particular attention. Moreover, we study many properties of Harnack equivalence in the general case.
\end{abstract}
 \maketitle
\section{Introduction and preliminaries}
Let $H$ be a complex Hilbert space and  $B(H)$ the set of all bounded linear operators on  $H$.
For $\rho>0$, we say that an operator $T\in B(H)$ admits  a \emph{unitary $\rho$-dilation} if there is a Hilbert space $\mathcal{H}$ containing $H$ as a closed subspace and a unitary operator $U\in B(\mathcal{H})$ such that 
\begin{equation}\label{eq:rhodil}
T^n=\rho P_H U^n|H,\quad
n\in\mathbb{N}^{\ast},
\end{equation}
where $P_H$ denotes the orthogonal projection onto the subspace $H$ in $\mathcal{H}$. 

In the sequel, we denote by $C_\rho (H),\ \rho>0$, the set of all operators in $B(H)$ which admit unitary $\rho$-dilations. A famous theorem due to B.~Sz.-Nagy \cite{SzN} the asserts that $C_1(H)$ is exactly the class of all contractions, i.e., operators
$T$ such that $\| T\|\leq 1$.  C.~A.~Berger \cite{Be} showed that
the class $C_2(H)$ is precisely the class of all operators $T\in
B(H)$ whose
the \emph{numerical radius}
\[ w(T)=\sup\{ |\langle Tx,x\rangle |:\ x\in H,\ \| x\|
=1\}\] is less or equal to one. In particular, the classes $C_\rho (H),\ \rho>0$,
provide a framework for simultaneous investigation of these two
important classes of operators.
Any operator $T\in C_\rho (H)$ is
\emph{power-bounded}:
\begin{equation}\label{eq:1-pb}
 \| T^n\|\leq \rho,\quad n\in\mathbb{N},
\end{equation}
moreover, its \emph{spectral radius}
\begin{equation}\label{eq:1-sp-rad}
r(T)=\lim_{n\rightarrow+\infty}\| T^n\|^{\frac{1}{n}}
\end{equation}
 is at most one. In \cite{SzNF1}, an example of a power-bounded operator which
is not contained in any of the classes $C_\rho (H),\ \rho>0$, is
given. However, J.~A.~R.~Holbrook \cite{H1} and J.~P.~Williams \cite{W},
independently, introduced the 
$\rho$-\emph{numerical radius}
(or the \emph{operator radii} ) of an operator $T\in B(H)$ by setting
\begin{equation}\label{eq:1-rad}
 w_\rho(T):=\inf\{ \gamma>0:\frac{1}{\gamma}T\in C_\rho (H)\}.
\end{equation}
Note that $w_1 (T)=\left\|T\right\|$, $w_2 (T)=w(T)$ and $ \lim_{\rho\rightarrow\infty}w_\rho(T)=r(T)$. Also, $T\in C_\rho (H)$ if and only if $w_\rho(T)\leq 1$, hence operators in $ C_\rho (H)$ are contractions with respect to the $\rho$-\emph{numerical radius}, and the elements of $ C_\rho (H)$ are called $\rho$-contractions. 

Some properties of the classes $C_\rho (H)$ become more clear (see for instance, \cite{CaF_1},\cite{CaF_2}, \cite{CaSu} and \cite{Cassier_2}) due to the use of the following operatorial $\rho$-kernel for  a bounded operator $T$ having its spectrum in the closed unit disc $\overline{\mathbb{D}}$, harmonic method in operator analysis introduced and first systematically developed in \cite{Cassier_1, CaF_1, CaF_2}:
\begin{equation}
K_{z}^{\rho} (T) =(I-\overline{z}T)^{-1}+(I-zT^{*})^{-1}+ (\rho-2)I, \quad (z\in \mathbb{D}).
\end{equation}

The $\rho$-kernels are related to $\rho$-contraction by the next result. An operator $T$ is in the class $C_\rho (H)$ if and only if, $\sigma(T)\subseteq \overline{\mathbb{D}}$ and
$K_{z}^{\rho} (T)\geq 0$ for any $z\in \mathbb{D}$  (see \cite{CaF_2}).

We say that $T_1$ is Harnack dominated by $T_0 $, if $T_0$ and $T_1$ satisfy one of the following equivalent conditions of the following theorem:
 \begin{theorem}\cite[ Theorem 3.1]{CaSu} \label{harnack} For $T_0 , T_1 \in C_\rho (H)$ and a constant $c\geq 1$, the following statements are equivalent:
\begin{enumerate}
	\item [(i)] $Re p(T_1) \leq c^{2} Re p(T_0)+ (c^{2}-1)(\rho -1) Re p(O_H)$, fo any polynomial $p$ with $Re p\geq 0$ on $\overline{\mathbb{D}}$.
	\item [(ii)] $Re p(rT_1)\leq c^{2} Re p(rT_0)+ (c^{2}-1)(\rho -1) Re p(O_H)$, fo any $r \in \left]0,1\right[$ and each  polynomial $p$ with $Re p\geq 0$ on $\overline{\mathbb{D}}$.
	\item [(iii)]  $K_{z}^{\rho} (T_1) \leq c^{2} K_{z}^{\rho} (T_0)$,  for all  $z\in \mathbb{D}$.
	\item [(iv)] $\varphi_{T_{1}} (g)\leq c^{2}\varphi_{T_{0}} (g)$ for any function $g\in C(\mathbb{T})$ such that $g\geq 0$ on $\mathbb{T}=\overline{\mathbb{D}} \setminus \mathbb{D}$.
	\item [(v)] If $V_i$ acting on $K_i \supseteq H$ is the minimal isometric $\rho$-dilation of $T_i$ $(i=0, 1)$, then there is an operator $S\in B(K_0 , K_1 )$ such that $S(H) \subset H$, $S|_{H}=I$, $SV_0 =V_{1}S$ and $\left\|S\right\|\leq c$.
\end{enumerate}
 \end{theorem}
 When $T_1$ is Harnack dominated by $T_0$ in $C_\rho (H)$ for some constant $c \geq 1$, we write $T_1\stackrel{H}{\underset{c}{\prec}}T_0$, or also $T_1{\stackrel{H}{\prec}} T_0$.
 The relation ${\stackrel{H}{\prec}}$ is a preorder relation in $C_\rho (H)$. The induced equivalent relation is called Harnack equivalence, and the associated classes are called the Harnack parts of $C_\rho (H)$. So, we say that $T_1$ and $T_0$ are Harnack equivalent if they belong to the same Harnack parts. In this later case, we write $T_1  {\stackrel{H}{\sim}} T_0$.
 
We say that an operator $T\in C_\rho (H)$ is a strict $\rho-$contraction if $w_{\rho} (T)<1$. In  \cite{Foias}  C. Foia\c{s} proved that  the Harnack parts of contractions containing the null operators $O_H$ consists of all strict contractions. More recently, G. Cassier and N. Suciu proved in \cite[Theorem 4.4]{CaSu} that the Harnack parts of $C_\rho (H)$ containing the null operators $O_H$ consists of all strict $\rho-$contractions. According to this fact the following natural question arises:

\textit{If $T$ an operator with $\rho$-numerical radius one, what can be said about the Harnack part of $T$?}

Recall that a $\rho$-contractions is similar to a contraction \cite{SzNF2}
but many properties are not preserved under similarity
(and an operator similar to a contraction is not necessarily a $\rho$-contraction!), in particular it is true for the numerical range properties. Thus, the study of Harnack parts for $\rho$-contractions cannot be deduced from the contractions case. Notice also that some properties are of different nature (see for instance Theorem \ref{spectral+torus} and Remark \ref{per_spect}).
 We  find a  few answers in the literature of the previous question, essentially  in the class of contractions  with norm one. In \cite{AnSuTi}, the authors have proved that if $T$ is either isometry or coisometry contraction then the Harnack part of $T$ is trivial (i.e. equal to $\{T\}$), and if $T$ is compact or $r(T)<1$, or normal and nonunitary, then its Harnack  part is not trivial in general. The authors have asked that  it seems interesting to give necessary and/ or sufficient condition for a contractions to have a trivial Harnack part.  It was proved in \cite{KSS} that the Harnack part of a contraction $T$ is trivial  if and only if $T$ is an isometry or a coisometry (the adjoint of an isometry), this  a response of the question posed by T. Ando and al. in the class  of contractions. Recently the authors  of \cite{BaTiSu} proved that maximal elements for the  Harnack domination in $C_1(H)$ are precisely the singular unitary operators and the minimal elements are isometries and coisometries.

This paper  is a continuation and refinement of the research treatment of  the Harnack domination in the general case of the $\rho-$contractions. Note that this treatment  yields certain useful properties and new techniques for studies of the Harnack parts of an operator with $\rho$-numerical radius one. More precisely,  
we show that two  $\rho$-contractions belong to the same Harnack parts have the same spectral values in $\mathbb{T}$. This property has several consequences and applications. In particular, it will be shown that if $T$ is a compact 
(i.e. $T \in \mathcal{K}(H)$) with $ w_\rho(T)=1$  and whose spectral radius is strictly less than one, then a $\rho$-contraction $S \in \mathcal{K}(H)$ is Harnack equivalent to $T$ if and only if $K_{z}^{\rho} (S)$ and $K_{z}^{\rho} (T)$ have the same kernel for all $z\in\mathbb{T}$. As a corollary, in the case of contractions, we show that if $T$ is a compact contraction with $ \Vert T\Vert=1$  and $r(T)<1$, then a contraction
$S \in \mathcal{K}(H)$ is Harnack equivalent to $T$ if and only if $I-S^*S$ and
$I-T^*T$ have the same kernel and $S$ and $T$ restricted to the kernel of $I-T^*T$ coincide. A nice application is the  description  of the Harnack part of the (nilpotent) Jordan block of size $n$. We also obtain precise results about the relationships between the closure of the numerical range and the Harnack domination for every $\rho \in \left[ 1,2\right] $. The case of $\rho=2$ plays a crucial role. We characterizes the weak stability of a $\rho-$contraction in terms of its minimal isometric $\rho-$dilation. The details of these basic facts are explained in Section \ref{sec:2}. In the last section we apply the results in the foregoing section to  describe the Harnack part of some nilpotent matrices with numerical radius one, in three cases: a nilpotent matrix of order two  in the two dimensional case,  a nilpotent matrix of order two  in $\mathbb{C}^n$ and a nilpotent matrix of order three in the three dimensional case. In particular, we show that in the first case the Harnack part is trivial, while in the  third case the
Harnack part is an orbit associated with the action of a group of unitary 
diagonal matrices.

\section{Main results}\label{sec:2}
\subsection{Spectral properties and Harnack domination}

We denote by $\Gamma(T)$ the set of complex numbers defined by $\Gamma(T)=\sigma(T) \cap \mathbb{T}$, where $\mathbb{T}=\overline{\mathbb{D}}\setminus\mathbb{D}$ is the unidimensional torus. In the following results, we prove that $\rho$-contractions belonging to the same Harnack parts have the same spectral values in the torus.

\begin{theorem}\label{spectral+torus}
Let $T_0 , T_1 \in C_{\rho}(H)$, ($\rho \geq 1$), if $T_1 {\stackrel{H}{\prec}} T_0$ then $\Gamma(T_1) \subseteq \Gamma(T_0)$.
\end{theorem}
\begin{proof}
 Let $T_0 , T_1 \in C_{\rho}(H)$ be such that $T_1 {\stackrel{H}{\prec}} T_0$. Then there exists  $c\geq 1$ such that 
 \begin{equation}
 K_{z}^{\rho} (T_1) \leq c^{2} K_{z}^{\rho} (T_0), \quad \text{ for all } z\in \mathbb{D},
 \end{equation} 
 so,  
\begin{align*}
K_{z}^{\rho} (T_1)& =(I-zT^{*}_{1})^{-1}[ \rho I+2(1-\rho) Re (\overline{z}T_1)+(\rho-2) \left|z\right|^{2} T^{*}_{1}T_{1}] (I-\overline{z}T_{1})^{-1}\\
& \leq c^{2} K_{z}^{\rho} (T_0), \quad \text{ for all } z\in \mathbb{D}.	
\end{align*}
Hence
\begin{equation}\label{inequ1}
 \rho I+2(1-\rho) Re (\overline{z}T_1)+(\rho-2) \left|z\right|^{2} T^{*}_{1}T_{1} \leq c^{2}(I-zT^{*}_{1})K_{z}^{\rho} (T_0)(I-\overline{z}T_{1}),
  \end{equation} 
for all  $z\in \mathbb{D}$. Now, let  $\lambda =e^{i\omega}\in \Gamma(T_1)\subseteq \sigma_{ap}(T_1)$, then there exists a sequence $(x_n)_{n\geq 0}$ of unit vectors such that $T_1 x_n -e^{i\omega}x_n =y_n$ converge to $0$. From the inequality \eqref{inequ1}, we derive
\begin{align*}
 \rho I+2(1-\rho)& Re (\overline{z}\left\langle T_1 x_n, x_n \right\rangle)+(\rho-2) \left|z\right|^{2} \left\|T_{1}x_n\right\|^{2} \\ &\leq c^{2} \left\langle K_{z}^{\rho} (T_0)(I-\overline{z}T_{1})x_n ,(I-\overline{z}T_{1})x_n\right\rangle\\
 & = c^{2} \left\langle K_{z}^{\rho} (T_0)[(1-\overline{z}e^{i\omega})x_n-\overline{z}y_n ] ,(1-\overline{z}e^{i\omega})x_n-\overline{z}y_n\right\rangle\\
 &= c^{2}\left|1-\overline{z}e^{i\omega}\right|^{2}\left\langle K_{z}^{\rho} (T_0)x_n ,x_n\right\rangle -c^{2}z  (1-\overline{z}e^{i\omega})\left\langle K_{z}^{\rho} (T_0)x_n ,y_n\right\rangle \\
 & - c^{2}\overline{z}(1-ze^{-i\omega})\left\langle K_{z}^{\rho} (T_0)y_n, x_n\right\rangle + c^{2}\left|z\right|^{2}\left\langle K_{z}^{\rho} (T_0)y_n ,y_n\right\rangle,
\end{align*} 
 for any $z\in \mathbb{D}$ and all $n\geq 0$. Note that 
 $$
 \left| \left\| T_1 x_n -e^{i\omega}x_n\right\|- \left\| x_n\right\|\right|\leq \left\| T_1 x_n \right\|\leq \left\| T_1 x_n -e^{i\omega}x_n\right\| +1.
 $$
Letting $n \rightarrow +\infty$, from the two previous inequalities we obtain 
 \begin{equation*}
 \rho +2(1-\rho) Re (\overline{z} e^{i\omega})+(\rho-2) \left|z\right|^{2}\leq c^{2}\left|1-\overline{z}e^{i\omega}\right|^{2} 
 \limsup_{n \rightarrow +\infty}\left\langle K_{z}^{\rho} (T_0)x_n ,x_n\right\rangle,
 \end{equation*}
 for any $z\in \mathbb{D}$. Then, if we take  $z=(1-t)e^{i\omega}$ with $0<t<1$, we get
 \begin{equation*}
 \rho +2(1-\rho) (1-t)+(\rho-2) (1-t)^{2}\leq c^{2}t^{2} \limsup_{n \rightarrow +\infty}\left\langle K_{(1-t)e^{i\omega}}^{\rho} (T_0)x_n ,x_n\right\rangle.
 \end{equation*}
 Assume that $e^{i\omega}\notin \Gamma(T_0)$, then $ K_{(1-t)e^{i\omega}}^{\rho} (T_0)$ is uniformly bounded in $ \left]0,1 \right[$, then there exists $\gamma >0$ such that 
  $$
  \rho +2(1-\rho) (1-t)+(\rho-2) (1-t)^{2}\leq \gamma c^{2}t^{2},
  $$
  which implies
 $$
   2t\leq (\gamma c^{2}+2-\rho) t^{2},
  $$  
  for all $t>0$, and hence
  $$
  2\leq (\gamma c^{2}+2-\rho) t.
  $$
  Now, we get a contradiction by letting $t \rightarrow 0$. Hence $e^{i\omega}\in \Gamma(T_0)$.
\end{proof}
From Theorem \ref{spectral+torus}, we also obtain the following result
\begin{corollary}\label{harnackequ}
If  $T_1$ and $T_0$ are Harnack equivalent in $C_{\rho}(H)$ then $\Gamma(T_1)= \Gamma(T_0)$.
\end{corollary}
Let $T\in B(H)$ and $E$ be a closed invariant subspace of $T$, ($T(E)\subset E$). Then $T\in B(E\oplus E^{\bot})$, has the following form:
$$T= \begin{pmatrix}
T_1 & R \\
0 & T_2
\end{pmatrix},$$
with $T_1\in B(E)$, $T_2\in B( E^{\bot})$ and $R$ is a bounded operator from $E^{\bot}$ to $E$.
 We denote by $\Gamma_p(T)=\sigma_p(T) \cap \mathbb{T}$ the  point spectrum of  $T \in C_{\rho}(H)$ in the unidimensional torus. 
\begin{theorem}\label{thm:spp}
Let $T_0 , T_1 \in C_{\rho}(H)$ ($\rho \geq 1$), if $T_1 {\stackrel{H}{\prec}} T_0$ then $\Gamma_p(T_1) \subseteq \Gamma_p(T_0)$ and $Ker(T_1 -\lambda I)\subseteq Ker(T_0 -\lambda I)$ for all  $\lambda \in \Gamma_p(T_1)$.
\end{theorem} 
For the proof of this theorem we need the following lemma.
\begin{lemma}\label{lem:2.4} Let $T \in C_{\rho}(H)$. Then 
$$ \Vert (I- \overline{\lambda}T) K_{z}^{\rho} (T)(I- \lambda T^*) \Vert \leq \rho(1+2\vert 1-\rho\vert+\vert\rho-2\vert\rho)(1+\rho\frac{\vert z-\lambda \vert}{1-\vert z \vert})^2, $$
for all $z\in \mathbb{D}$  and $\lambda\in \overline{\mathbb{D}}$.
\end{lemma}
\begin{proof}Let $z\in \mathbb{D}$  and $\lambda\in \overline{\mathbb{D}}$, we have
\begin{align*}
(I-zT^{*})^{-1} (I-\lambda T^*)& =(I-zT^{*})^{-1}[I-zT^{*}+ (z-\lambda) T^{*}]\\
& = I+ (z-\lambda)\sum^{+\infty}_{n=0} z^{n}T^{*n+1}.	
\end{align*}
Then by \eqref{eq:1-pb},
$$\Vert(I-zT^{*})^{-1} (I-\lambda T^*)\Vert \leq  1+\rho\frac{\vert z-\lambda \vert}{1-\vert z \vert}.$$
 Taking into account this inequality and the fact that $$K_{z}^{\rho} (T_1) =(I-\overline{z}T_{1})^{-1}[ \rho I+2(1-\rho) Re (\overline{z}T_1)+(\rho-2) \left|z\right|^{2}T_{1}T^{*}_{1}](I-zT^{*}_{1})^{-1},$$ we obtain the desired inequality.
\end{proof}
\begin{proof}[Proof of Theorem \ref{thm:spp}] Let  $\lambda \in \Gamma_p(T_1)$. Then  the operator $T_1 \in C_{\rho}(H)$ on $Ker(T_1 -\lambda I)\oplus Ker(T_1 -\lambda I)^{\bot}$ takes the form
$$T_1= \begin{pmatrix}
 \lambda & C \\
0 & \tilde{T_1}
\end{pmatrix}.$$
Since $\vert\lambda\vert=1$, by using Proposition \cite[Proposition 3.]{CaZe} we can see 
that $C=0$. Thus, we have
$$K_{z}^{\rho}(T_1)= \begin{pmatrix}
 \frac{\rho +2(1-\rho) Re (\overline{\lambda}z)+(\rho-2) \left|\lambda\right|^{2} \left|z\right|^{2}}{\vert 1- \overline{\lambda}z\vert^{2}}I& 0 
\\
0 & K_{z}^{\rho}(\tilde{T_1})
\end{pmatrix},$$
Now, if  $T_0  \in C_{\rho}(H)$ be such that $T_1 {\stackrel{H}{\prec}} T_0$, then there exists  $c\geq 1$ such that 
 \begin{equation*}
 K_{z}^{\rho} (T_1) \leq c^{2} K_{z}^{\rho} (T_0), \quad \text{ for all } z\in \mathbb{D},
 \end{equation*} 
 Let $x\in Ker(T_1 -\lambda I)$ and $y\in  Im(T_{0}^{*} -\lambda I)$. The  Cauchy-Schwarz inequality yields
 $$ \left|\left\langle K_{z}^{\rho} (T_1)x ,y\right\rangle\right|^{2}  \leq  c^2 \left\langle K_{z}^{\rho} (T_1)x ,x\right\rangle \left\langle K_{z}^{\rho} (T_0)y ,y\right\rangle.$$
 We derive
 $$ \frac{\rho I+2(1-\rho) Re (\overline{\lambda}z)+(\rho-2) \left|\lambda\right|^{2} \left|z\right|^{2}}{\vert 1- \overline{\lambda}z\vert^{2}} \left|\left\langle x ,y\right\rangle\right|^{2}  \leq  c^2 \left\langle K_{z}^{\rho} (T_0)y ,y\right\rangle
 \Vert x\Vert ^2.$$
 Since $y\in  Im(T_{0}^{*} -\lambda I)$, there exits $u\in H$ such that $y=(I -\lambda T_{0}^{*} ) u$. By Lemma \ref{lem:2.4}, we have
 \begin{align*}
\left\langle K_{z}^{\rho} (T_0)y ,y\right\rangle & =\left\langle ( I-\overline{\lambda} T_0) K_{z}^{\rho} (T_0)(I -\lambda T_{0}^{*} ) u ,(I -\lambda T_{0}^{*} ) u\right\rangle \\
& \leq \rho(1+2\vert 1-\rho\vert+\vert\rho-2\vert\rho)(1+\rho\frac{\vert z-\lambda \vert}{1-\vert z \vert})^2 \Vert u\Vert ^2.	
\end{align*}
Let $z=r\lambda $, with $0<r<1$. Then 
$$
\frac{\vert\rho +2(1-\rho) r+(\rho-2) r^{2}\vert }{(1-r)^{2}} \left|\left\langle x ,y\right\rangle\right|^{2}   \leq c^2\rho(1+2\vert 1-\rho\vert+\vert\rho-2\vert\rho)(1+\rho)^2 \Vert u\Vert ^2 \Vert x\Vert ^2.$$
This implies
$$
 \vert\rho +2(1-\rho) r+(\rho-2) r^{2}\vert \vert\left\langle x ,y\right\rangle\vert^{2}   \leq c^2 (1-r)^2\rho(1+2\vert 1-\rho\vert+\vert\rho-2\vert\rho)(1+\rho)^2 \Vert u\Vert ^2 \Vert x\Vert ^2 .$$
By letting $r$ to $1$, it follows that $\left\langle x ,y\right\rangle= 0$, and hence $x\in Im(T_{0}^{*} -\lambda I)^{\bot}= Ker(T_{0} -\lambda I)$. So, $\Gamma_p(T_1) \subseteq \Gamma_p(T_0)$ and $Ker(T_1 -\lambda I)\subseteq Ker(T_0 -\lambda I)$.
\end{proof}
\begin{remark} By Theorem \ref{thm:spp}, if  $I_H {\stackrel{H}{\prec}} T$ on $C_{\rho}(H)$, ($\rho \geq 1$) then $T=I_H$. This means that $I_H$ is a maximal element  for the Harnack domination on $C_{\rho}(H)$ and its Harnack part is trivial,  for all $\rho \geq 1$.
\end{remark}
From Theorem \ref{thm:spp}, we also obtain the following result
\begin{corollary}\label{harnackequspp}
If  $T_1$ and $T_0$ are Harnack equivalent in $C_{\rho}(H)$ then $\Gamma_p (T_1)= \Gamma_p (T_0)$ and $Ker(T_1 -\lambda I)= Ker(T_0 -\lambda I)$ for all $\lambda \in  \Gamma_p (T_0)$.
\end{corollary}
\begin{remark}\label{per_spect} After the authors have obtained
Theorem \ref{spectral+torus}, they have learned that C. Badea, D. Timotin and L. Suciu \cite{BaTiSu} have proved using an other method that, in the case of contractions ($\rho=1$), the domination suffices for the equality of the  point spectrum in the torus. But in the case of $\rho>1$ the inclusion in Theorem \ref{thm:spp} may be strict, for instance, we have 
\begin{itemize}
\item For  $\rho>1$, we have $0_H\stackrel{H}{\underset{c}{\prec}}I$ in $C_{\rho}(H)$  with  
$c=\sqrt{\dfrac{\rho}{\rho-1}}$.
\item For $1<\rho$, the operator $T$ defined on $\mathbb{C}^2$ by
 $T =
 \begin{pmatrix}
0 &\rho \\
0 & 0
\end{pmatrix}$ satisfies $T\stackrel{H}{\underset{c}{\prec}}I$ in $C_{\rho}(H)$ with $c=\sqrt{\dfrac{2\rho}{\rho-1}}$.
\end{itemize}
\end{remark}
\begin{corollary}\label{coro:he}Let $T_0 , T_1 \in C_{\rho}(H)$ ($\rho \geq 1$) such that $\Gamma(T_0)=\Gamma_p(T_0)$. Then $T_0$ and  $T_1$  are Harnack equivalent in $C_{\rho}(H)$ if and only if $T_0 =U \oplus \tilde{T_0}$ and $T_1 =U \oplus \tilde{T_1}$ on $H =E \oplus E^{\bot}$, where $E=\oplus_{\lambda \in  \Gamma_p (T_0)}Ker(T_0 -\lambda I)=\oplus_{\lambda \in  \Gamma_p (T_1)}Ker(T_1 -\lambda I )$, $U$ is an unitary diagonal operator on $E$ and $\tilde{T_0}$ and  $\tilde{T_1}$  are Harnack equivalent in $C_{\rho}(E^{\bot})$.
\end{corollary}
\begin{proof}First we prove that if $\lambda, \mu \in  \Gamma_p (T_0)$, then $Ker(T_0 -\lambda I)\bot Ker(T_0 -\mu I)$ for $\lambda\neq \mu$.  Let $x\in Ker(T_1 -\lambda I)$ and $y\in  Ker(T_{0} -\mu I)$. Then
\begin{align*}
\left\langle K_{z}^{\rho} (T_0)x ,y\right\rangle & =\left\langle ( (I-\overline{z}T)^{-1}+(I-zT^{*})^{-1}+ (2-\rho)I) x ,y\right\rangle
 \\
& =\frac{1}{1- \overline{z}\lambda}\left\langle x ,y\right\rangle + \frac{1}{1- z\overline{\mu}}\left\langle x ,y\right\rangle +(2-\rho ) \left\langle x ,y\right\rangle.	
\end{align*}
By Cauchy-Schwarz inequality 
$$
\vert\left\langle K_{z}^{\rho} (T_0)x ,y\right\rangle\vert^2  \leq
\left\langle K_{z}^{\rho} (T_0)x ,x\right\rangle \left\langle K_{z}^{\rho} (T_0)y ,y\right\rangle.
$$
Thus
$$
\vert\frac{1}{1- \overline{z}\lambda}+\frac{1}{1- z\overline{\mu}}+(2-\rho)\vert^2
\vert\left\langle x ,y\right\rangle\vert^2\leq $$ $$\frac{(\rho +2(1-\rho) Re (\overline{\lambda}z)+(\rho-2) \left|z\right|^{2})(\rho +2(1-\rho) Re (\overline{\mu}z)+(\rho-2) \left|z\right|^{2})}{\vert 1- \overline{z}\lambda\vert^{2}\vert 1- \overline{\mu}z\vert^{2}}\Vert x\Vert^2 \Vert y\Vert^2.
$$
So
$$
\vert 1+\frac{1- \overline{z}\lambda}{1- z\overline{\mu}}+(2-\rho)(1- \overline{z}\lambda) \left|z\right|^{2}\vert^2
\vert\left\langle x ,y\right\rangle\vert^2\leq $$ $$\frac{(\rho +2(1-\rho) Re (\overline{\lambda}z)+(\rho-2) \left|z\right|^{2})(\rho +2(1-\rho) Re (\overline{\mu}z)+(\rho-2) \left|z\right|^{2})}{\vert 1- \overline{\mu}z\vert^{2}}\Vert x\Vert^2 \Vert y\Vert^2.
$$
By taking $z$ to $\lambda$, we get $\left\langle x ,y\right\rangle =0$. By \cite[Corollary  4.]{CaZe} the subspace $E$ reduces $T_0$ and $T_1$. 
\end{proof}
\begin{example}Recall  that an operator $T\in B(H)$  is called to be quasi-compact
operator (or quasi-strongly completely continuous in the terminology of \cite{YoKa})  if there exists a compact operator $K$ and an integer $m$ such that 
$\left\Vert T^{m}-K\right\Vert <1$. Since  every operator  $T \in C_{\rho}(H)$ ($\rho \geq 1$) is power-bounded, by \cite[Theorem 4]{YoKa}; if $T \in C_{\rho}(H)$ ($\rho \geq 1$) is a quasi-compact operator then $\Gamma(T)=\Gamma_p(T)$ and contains a finite  number of eigenvalues and each of them is of finite multiplicity. Now if we assume that $T ,S $    are two quasi-compact operators which are Harnack equivalent in $C_{\rho}(H)$, ($\rho \geq 1$), then $S =U \oplus \tilde{S}$ where $U$ is an unitary diagonal operator on $E=\oplus_{i=1}^{k}Ker(T -\lambda_i I )$ and $\tilde{S}$ is  Harnack equivalent to $0$ in $C_{\rho}(E^{\bot})$.
\end{example}
\begin{corollary}\label{coro:hecn}Let $T \in C_{\rho}(H)$ ($\rho\geq 1$) be a compact normal operator  with $w_{\rho}(T)=1$. If the operator $S \in C_{\rho}(H)$ is Harnack equivalent to $T$, then for all  $\lambda \in  \Gamma_p (T)$, $S_{|E}=T_{|E}$ where $E=\oplus_{\lambda \in  \Gamma_p (T)}Ker(T -\lambda I)$, $E$ is a reducing subspace for $S$ and $S_{E^{\bot}}$ is Harnack equivalent to $0$, i.e. $w_{\rho} (S_{E^{\bot}})<1$.
\end{corollary}
\begin{proof}By Corollary \ref{coro:he}, for all  $\lambda \in  \Gamma_p (T)$, we have $T =U \oplus \tilde{T}$ and $S =U \oplus \tilde{S}$ on $E \oplus E^{\bot}$,  where $E=\oplus_{\lambda \in  \Gamma_p (T)}Ker(T -\lambda I)$ and  $\tilde{T}$ and  $\tilde{S}$  are Harnack equivalent in $C_{\rho}(E^{\bot})$. Since $T \in C_{\rho}(H)$ is a compact normal operator  we also have
$$ w_{\rho} (\tilde{T})=\sup \{ \vert\lambda\vert, \lambda \in \sigma(T)\setminus  \Gamma_p (T)\}<1.$$
This means that $\tilde{T}$ and $\tilde{S}$ are Harnack equivalent to $0$.
\end{proof}

In the following proposition, we prove that the $\rho$-contractions belong to the same Harnack parts have the same kernel for their operatorial $\rho$-kernels.
\begin{proposition}\label{prop2.3}
Let $T_0 , T_1 \in C_{\rho}(H)$. If  $T_1$ and $T_2$ are Harnack equivalent in $C_{\rho}(H)$ then $Ker K_{z}^{\rho} (T_1)=Ker K_{z}^{\rho} (T_2)$ for all $z\in \mathbb{D}$.
\end{proposition}
\begin{proof}
Since  $T_1{\stackrel{H}{\sim}} T_2$, then by Theorem \ref{harnack}, there exist $\alpha, \beta>0$ $(\alpha\leq 1$, $\beta \geq 1$) such that 
\begin{equation}\label{equ2222223}\alpha K_{z}^{\rho} (T_1) \leq  K_{z}^{\rho} (T_2)\leq \beta K_{z}^{\rho} (T_1), \quad \text{ for all } z\in \mathbb{D}.
\end{equation}
If $x\in Ker K_{z}^{\rho} (T_1)$, then by the right side of the inequality \eqref{equ2222223}, we also have,
$$ 0\leq \left\langle K_{z}^{\rho} (T_2)x ,x\right\rangle \leq  \beta \left\langle K_{z}^{\rho} (T_1)x ,x\right\rangle=0, $$
this implies that $\left\|\sqrt{K_{z}^{\rho} (T_2)}x\right\|=0$, so $K_{z}^{\rho} (T_2)x=0$, hence $Ker K_{z}^{\rho} (T_1)\subseteq Ker K_{z}^{\rho} (T_2)$ for all $z\in \mathbb{D}$. The converse inclusion holds from  the left-side of the inequality \eqref{equ2222223}.
\end{proof}
\begin{proposition} If $w(T)=1$ and $\Gamma(T)$ is  empty then there exists $z_0 \in \mathbb{T}$ such that $K_{z_0}^{2} (T)$ is not invertible.
\end{proposition}
\begin{proof} Since $w(T)=1$, there exists a sequence $(x_n)_{n\geq 0}$ of unit vectors such that $\left\langle Tx_n , x_n\right\rangle$ converge to $z_0=e^{i\omega}\in \mathbb{T}$. Set $y_n = (I -e^{-i\omega}T)x_n $, then $\left\|y_n\right\|$ not converge to $0$. If not we have $e^{i\omega}\in \sigma(T)$, this contradicts that $\Gamma(T)$ is  empty. Thus, we may suppose that $\left\|y_n\right\|\rightarrow l>0$ and we have
\begin{align*}
 \left\langle K_{e^{i\omega}}^{2} (T)y_n ,y_n\right\rangle 
 &= 2 \left\langle (I-Re(e^{-i\omega}T))x_n ,x_n\right\rangle\\
  &= 2(1-  Re(e^{-i\omega}\left\langle Tx_n ,x_n\right\rangle)),
\end{align*} 
hence
$$ \left\langle K_{e^{i\omega}}^{2} (T)\frac{y_n}{\left\|y_n\right\|} ,\frac{y_n}{\left\|y_n\right\|}\right\rangle=2 \frac{(1-  Re(e^{-i\omega}\left\langle Tx_n ,x_n\right\rangle))}{\left\|y_n\right\|^{2}} \rightarrow 0.$$
This implies that $0\in \sigma_{ap}(\sqrt{K_{e^{i\omega}}^{2} (T)})$ and hence $0\in \sigma_{ap}(K_{e^{i\omega}}^{2} (T))$.
\end{proof}
\subsection{Numerical range properties and Harnack domination}

Firstly, we give a proposition which is useful in this subsection.
\begin{proposition}\label{rho1<<rho2}
Let $T_0 , T_1 \in C_{\rho_1}(H)$ and $\rho_2 \geq \rho_1$. Then we have
\begin{enumerate}
	\item [(i)] If $T_1\stackrel{H}{\underset{c}{\prec}}T_0$ in $C_{\rho_1}(H)$, then $T_1\stackrel{H}{\underset{c}{\prec}}T_0$ in $C_{\rho_2}(H)$.
	\item [(ii)] If $T_1\stackrel{H}{\underset{c}{\sim}}T_0$ in $C_{\rho_1}(H)$, then $T_1\stackrel{H}{\underset{c}{\sim}}T_0$ in $C_{\rho_2}(H)$.
\end{enumerate}
\end{proposition}
\begin{proof}
(i) Since the $C_{\rho}$ classes increase with $\rho$, the two operators $T_0$ and $T_1$
belong to $C_{\rho_2}(H)$. From Theorem \ref{harnack}, we know that there exists $c \geq 1$ such that $K_{z}^{\rho_1} (T_1) \leq c^{2} K_{z}^{\rho_1} (T_0)$ for all $z \in \mathbb{D}$. As $c\geq 1$, it yields to
$$
K_{z}^{\rho_2} (T_1)=K_{z}^{\rho_1} (T_1)+(\rho_2-\rho_1)I
\leq c^{2} \left[ K_{z}^{\rho_1} (T_0)+(\rho_2-\rho_1)I\right]=c^{2}K_{z}^{\rho_2} (T_0)  .$$
Using again Theorem \ref{harnack}, we obtained the desired conclusion.
 
The assertion (ii) is a direct consequence of (i). 
\end{proof}

Let $T \in B(H)$, we denote by $W(T)$ the numerical range of $T$ which is the set given by
$$
W(T)=\left\lbrace \left\langle Tx , x\right\rangle; x \in H, \Vert x \Vert=1\right\rbrace. 
$$
The following result give relationships between numerical range and Harnack domination.
\begin{theorem}\label{thm:numr}Let $T_0, T_1 \in C_{\rho}(H)$ with $1 \leq \rho
\leq 2$, then we have:
\begin{enumerate}
    \item [(i)] Assume that $\rho=1$ and $T_1{\stackrel{H}{\prec}} T_0$, then $\overline{W(T_0)}\cap \mathbb{T} =
	\overline{W(T_1)}\cap \mathbb{T}$.
	\item [(ii)] Suppose that $1 < \rho
\leq 2$, $T_1{\stackrel{H}{\prec}} T_0$ and $\Gamma(T_0)=\emptyset$, then $\overline{W(T_0)}\cap \mathbb{T} \subseteq 
	\overline{W(T_1)}\cap \mathbb{T}$.
	\item [(iii)] If $T_1{\stackrel{H}{\sim}} T_0$, then
	$\overline{W(T_0)}\cap \mathbb{T} = 
	\overline{W(T_1)}\cap \mathbb{T}$.
\end{enumerate}
\end{theorem}
\begin{proof}
(i) Let $\lambda=e^{i\omega} \in \overline{W(T_0)}\cap \mathbb{T}$, then there exists a sequence 
$(x_n)$ of unit vectors such that 
$\left\langle T_0x_n , x_n\right\rangle \longrightarrow \lambda$. 
We have for some $c\geq 1$, $0\leq K_{r, \theta} (T_1) \leq c^{2} K_{r, \theta}(T_0)$ for all $z \in \mathbb{D}$. Multiplying these inequalities by the nonnegative function
$1-Re(\overline{\lambda} e^{i\theta})$, integrating with respect to the Haar measure $m$ and letting $r$ to $1$, we get
$0 \leq I-Re(\overline{\lambda} T_1)\leq c^{2}\left[ I-Re(\overline{\lambda} T_0)\right]  $. We deduce that $1-Re(\overline{\lambda} \left\langle T_1 x_n,
x_n\right\rangle) \longrightarrow 0$. Since $\left\langle T_1 x_n,
x_n\right\rangle $ belongs to the closed unit disc, it forces $\left\langle T_1 x_n,
x_n\right\rangle \longrightarrow \lambda $. Hence $\overline{W(T_0)}\cap \mathbb{T} 
\subseteq \overline{W(T_1)}\cap \mathbb{T}$. Now , let $\lambda \in 
\overline{W(T_1)}\cap \mathbb{T}$, then there exists a sequence 
$(y_n)$ of unit vectors such that 
$\left\langle T_1 y_n , y_n\right\rangle \longrightarrow \lambda$. As $T_1$ is a contraction, it follows that
$1 = \lim \vert \left\langle T_1 y_n , y_n\right\rangle \vert
\leq \underline{\lim}\Vert T_1 y_n \Vert \leq \overline{\lim}\Vert T_1 y_n \Vert
\leq 1$, thus $\Vert T_1 y_n \Vert \longrightarrow 1$. 
It implies $\Vert T_1 y_n - \lambda y_n \Vert^2=\Vert T_1 y_n \Vert^2
-2Re(\overline{\lambda}\left\langle T_1 y_n , y_n\right\rangle )+1\longrightarrow 0
$. Consequently, we have $\lambda \in \Gamma(T_1)$, by using
Theorem \ref{spectral+torus} we see that $\lambda \in \Gamma(T_0) \subseteq \overline{W(T_0)}\cap \mathbb{T}$. So we get the desired equality.  

(ii) Tacking into account Proposition \ref{rho1<<rho2}, it suffices to treat the case 
where $\rho=2$.
Let $\lambda=e^{i\omega} \in \overline{W(T_0)}\cap \mathbb{T}$, then there exists a sequence 
$(x_n)$ of unit vectors such that 
$\left\langle Tx_n , x_n\right\rangle \longrightarrow \lambda$. Set
$y_n=(I-e^{-i\omega}T_0)x_n$, since $\Gamma(T_0)=\emptyset$ we necessarily have
$\gamma=\inf\{\Vert y_n\Vert; n\geq 0\}>0$. Tacking $u_n=y_n /\Vert y_n \Vert $, we can see that
\begin{align*}
\left\langle K^{2}_{e^{i\omega}}(T_0)u_n , u_n\right\rangle
&=\frac{2}{\Vert y_n\Vert^2} \left\langle (I- Re(e^{-i\omega}T_0))x_n , x_n\right\rangle\\
&\leq \frac{2}{\gamma^2}  \left\langle (I- Re(e^{-i\omega}T_0))x_n , x_n\right\rangle\longrightarrow 0.
\end{align*}
Since $T_1{\stackrel{H}{\prec}} T_0$, there exists $c \geq 1$ such that
\begin{equation}\label{t0ht1}
 K_{z}^{2} (T_1) \leq c^{2} K_{z}^{2} (T_0), \quad \text{ for all } z\in \mathbb{D}.
 \end{equation} 
On the one hand, if $\lambda \in \Gamma(T_1)$ we have obviously $\lambda \in 
\overline{W(T_1)}$.
On the other hand, if $\lambda \notin \Gamma(T_1)$ we can extended \eqref{t0ht1} at
$z=\lambda$ and we get
$$
0 \leq \left\langle K^{2}_{e^{i\omega}}(T_1)u_n , u_n\right\rangle
\leq c^{2}
\left\langle K^{2}_{e^{i\omega}}(T_0)u_n , u_n\right\rangle
\longrightarrow 0,
$$
hence $\left\langle K^{2}_{e^{i\omega}}(T_1)u_n , u_n\right\rangle \longrightarrow 0$.
Observe that $\inf\{\Vert (I- Re(e^{-i\omega}T_1))^{-1}u_n \Vert; n\geq 0\}\geq 
\frac{1}{3}$.
Set $v_n=(1/\Vert (I- e^{-i\omega}T_1)^{-1}u_n \Vert)(I- Re(e^{-i\omega}T_1))^{-1}u_n$, we obtain
$$
\left\langle (I- Re(e^{-i\omega}T_1))v_n , v_n\right\rangle
\leq \frac{9}{2}  \left\langle K^{2}_{e^{i\omega}}(T_1)u_n , u_n\right\rangle\longrightarrow 0.
$$
We deduce that $ \left\langle Re(e^{-i\omega}T_1) v_n , v_n\right\rangle \longrightarrow 1$.
As $T_1 \in C_2(H)$, it yields to:
$$
1 \geq \vert \left\langle T_1 v_n , v_n\right\rangle \vert^2
=\vert \left\langle Re(e^{-i\omega}T_1) v_n , v_n\right\rangle \vert^2
+\vert \left\langle Im(e^{-i\omega}T_1) v_n , v_n\right\rangle \vert^2,
$$
and we derive successively that $\left\langle Im(e^{-i\omega}T_1) v_n , v_n\right\rangle \longrightarrow 0$ and $\left\langle T_1 v_n , v_n\right\rangle 
\longrightarrow \lambda$. Thus $\lambda \in \overline{W(T_1)}\cap \mathbb{T} $ and
it ends the proof of (i).

(iii) As before, we may suppose that $\rho=2$. Assume that $T_1{\stackrel{H}{\sim}} T_0$ and $\lambda \in \overline{W(T_0)}\cap \mathbb{T}$. By Corollary \eqref{harnackequ}, we have $\Gamma (T_0)=\Gamma (T_1)$. So, if $\lambda\in \Gamma (T_0)$ then  $\lambda\in \overline{W(T_1)}\cap \mathbb{T}$. Now, if $\lambda\notin \Gamma (T_0)$, we proceed as in the second item  (ii) to prove that $\lambda\in \overline{W(T_1)}\cap \mathbb{T}$. Interchanging the roles of $T_0$ and $T_1$ gives the desired equality.
\end{proof}
\begin{remark}\label{num-range} 

(1) The condition $\Gamma(T_0)=\emptyset$, in (ii), cannot be relaxed. In fact, we have $T_1= 0_H\stackrel{H}{\underset{c}{\prec}}I=T_0$ in $C_{\rho}(H)$ ($1<\rho \leq 2$) with   $c=\sqrt{\dfrac{\rho}{\rho-1}}$ but  $\overline{W(T_0)}\cap \mathbb{T}=\{1\}$ and $\overline{W(T_1)}\cap \mathbb{T}=\emptyset$.

(2) When $T$ is a contraction, we have $\overline{W(T)}\cap \mathbb{T}=\Gamma(T)$
(see for instance the end of the proof of (i)). So, the assertion (i) of Theorem
\ref{thm:numr} restore, in the case of domination, the equality of the  point spectrum in the torus obtained by C. Badea, D. Timotin and L. Suciu in \cite{BaTiSu} by another way.

\end{remark}
\begin{corollary} Let $T_0 \in C_{\rho}(H)$ with $1\leq \rho \leq 2$. If $\overline{W(T_0)}=\overline{\mathbb{D}}$, and satisfies $\Gamma(T_0)=\emptyset$ when $\rho \neq 1$, then $\overline{W(T_1)}=\overline{\mathbb{D}}$ for every $T_1 \in C_{\rho}(H)$ such that $T_1\stackrel{H}{{\prec}} T_0$. Furthermore, in the case of Harnack equivalence, we have $\overline{W(T_1)}=\overline{\mathbb{D}}$ as soon as 
$\overline{W(T_0)}=\overline{\mathbb{D}}$. 
\end{corollary}
\begin{proof} By Theorem \ref{thm:numr}, Proposition \ref{rho1<<rho2} and the convexity theorem of Toeplitz-Hausdorff, we obtain the desired conclusions.
\end{proof}

\subsection{Harnack parts in the space of compact operators}

Denote by $\mathcal{K}(H)$ the set of all compact operators. We have
\begin{theorem}\label{thm:hpc1}Let $T \in C_{\rho}(H)\cap \mathcal{K}(H)$ with $w_{\rho}(T)=1$  and  $\Gamma_p (T)$ is empty. Then  $S \in C_{\rho}(H)\cap \mathcal{K}(H)$ is Harnack equivalent to $T$ if and only if $ker (K_{z}^{\rho} (S))=ker (K_{z}^{\rho} (T))$ for all $z\in \mathbb{T}$.
\end{theorem}
For the proof of this theorem we need the following lemma.
\begin{lemma}\label{lem:2.5} Let  $T \in C_{\rho}(H)\cap \mathcal{K}(H)$ as in the previous theorem . If $\Vert K_{z}^{\rho}(T)\Vert=\lambda_{1}(z)\geq \lambda_{2}(z)\geq \ldots \geq \lambda_{n}(z)\geq \ldots$ are the  eigenvalues of $K_{z}^{\rho}(T)$, arranged in decreasing order, then the mapping $z \mapsto \lambda_{n}(z)$ is continuous on $\overline{\mathbb{D}}$, for all $n$.
\end{lemma}
\begin{proof}Let $R$ such that $rang(R)<n$. We have
$$\lambda_{n}(z)\leq \Vert K_{z}^{\rho}(T)-R\Vert \leq \Vert K_{z}^{\rho}(T)-K_{z'}^{\rho}(T\Vert+ \Vert K_{z'}^{\rho}(T)-R\Vert.$$
Hence
$$\lambda_{n}(z)\leq   \Vert K_{z}^{\rho}(T)-K_{z'}^{\rho}(T)\Vert+\lambda_{n}(z').$$
By interchanging $z$ by $z'$, we get
\begin{equation}\label{lambdan}
\vert\lambda_{n}(z)-\lambda_{n}(z')\vert\leq   \Vert K_{z}^{\rho}(T)-K_{z'}^{\rho}(T\Vert.
\end{equation}
\end{proof}
\begin{proof}[Proof of Theorem \ref{thm:hpc1}] Let $T,S \in C_{\rho}(H)\cap \mathcal{K}(H)$ such that   $T{\stackrel{H}{\sim}} S$. Since $\Gamma_p (T)$ is empty, by Corollary \ref{harnackequ}, the operators $T$ and $S$ not admits eigenvalues in $\mathbb{T}$. Hence, $K_{z}^{\rho} (T)$ and  $ K_{z}^{\rho}(S)$ are uniformly bounded in $\mathbb{D}$ and may be extended to a positive operators on  $\overline{\mathbb{D}}$. Furthermore, if we proceed as in the proof of   Proposition \ref{prop2.3}, we deduce that  $Ker K_{z}^{\rho} (T)=Ker K_{z}^{\rho} (S)$ for all $z\in \mathbb{T}$.

Conversely, Let $E_{T}(z)=ker (K_{z}^{\rho} (T))$ and $E_{S}(z)=ker (K_{z}^{\rho} (S))$, both $K_{z}^{\rho} (T)$ and $K_{z}^{\rho} (S)$ on $H = E_{T}(z)\oplus E_{T}(z)^{\bot}$, take the following forms
 $$K_{z}^{\rho}(T)= \begin{pmatrix}
 0 & 0 
\\
0 & \tilde{K}_{z}^{\rho}(T)
\end{pmatrix} \quad \quad \text{ and } \quad \quad K_{z}^{\rho}(S)= \begin{pmatrix}
0 & 0 
\\
0 & \tilde{K}_{z}^{\rho}(S)
\end{pmatrix}.$$
Denote by  $\lambda_{1}^{T}(z), \lambda_{2}^{T}(z), \ldots, \lambda_{n}^{T}(z), \ldots$ are the  eigenvalues of $K_{z}^{\rho}(T)$, arranged in decreasing order $\lambda_{1}^{T}(z)=\Vert K_{z}^{\rho}(T)\Vert\geq \lambda_{2}^{T}(z)\geq \ldots \geq \lambda_{n}^{T}(z)\geq \ldots$.
We put $\lambda^{T}(z)=\inf_{n\geq 1}\lambda_{n}^{T}(z)$. We claim that $\lambda^{T}(z) >0$. Indeed, if we assume that there exist $z_0$ such that $\lambda^{T}(z_0)=0$, then there exists a sequence $(x_n)_n$ in $E_{T}(z_0)^{\bot}$ with $\Vert x_n \Vert=1$ such that 
$$\lambda_{n}^{T}(z)=\left\langle K_{z_{0}}^{\rho}(T) x_n , x_n \right\rangle = \left\langle \tilde{K}_{z_{0}}^{\rho}(T) x_n , x_n \right\rangle = \rho +\left\langle R_{z_{0}}^{\rho}(T)x_n , x_n \right\rangle,$$

with $R_{z_{0}}^{\rho} (T)=\sum^{+\infty}_{n=1} \overline{z_{0}}^n T^n + \sum^{+\infty}_{n=1} z_{0}^{n} T^{*n}$. Since $T$ is compact operator with $r(T)<1$, both of the series  $\sum^{+\infty}_{n=1} \overline{z_{0}}^n T^n $ and  $\sum^{+\infty}_{n=1} z_{0}^{n} T^{*n}$ are converge to a compact operator in the operator norm, so $R_{z_{0}}^{\rho} (T)$ is compact. There exist a subsequence $(x_{j(n)})$ of $(x_n)_n$ such that $(x_{j(n)})$ converges to some $x\in E_{T}(z_0)^{\bot}$ in the weak star topology, this implies that $R_{z_{0}}^{\rho}(T)x_{j(n)} \longrightarrow R_{z_{0}}^{\rho}(T)x$  strongly and 
$$0=\rho+\left\langle R_{z_{0}}^{\rho}(T)x , x \right\rangle\geq \Vert x \Vert^2+\left\langle R_{z_{0}}^{\rho}(T)x , x \right\rangle=\left\langle K_{z_{0}}^{\rho}(T)x , x \right\rangle\geq 0,$$
so $x\in E_{T}(z_0)$. If $x=0$, then  $R_{z_{0}}^{\rho}(T)x_k \longrightarrow 0$  strongly and $\lambda_{k}^{T}(z_{0}) \rightarrow 1$, which is a contradiction  with $\lambda_{k}^{T}(z_{0}) \rightarrow 0$. Then $x\neq 0$ but $x\in E_{T}(z_{0})\cap E(z_{0})^{\bot}$ this is again a contradiction and $\lambda^{T}(z_0)$ must be strictly positive. On the other hand $(\lambda_{n}^{T}(z))_n$ is a decreasing bounded below  sequence so it converge to $\lambda^{T}(z)$, furthermore, since, by Lemma  \ref{lem:2.5}, the mapping $z\mapsto \lambda_{n}^{T}(z)$  is continuous, then by letting $n$ to $+\infty$ in \eqref{lambdan}, we deduce that the mapping $z\mapsto \lambda^{T}(z)$ is also continuous and has a minimum in $\mathbb{T}$ denoted by  $m(T) =\inf_{\lambda \in \mathbb{T}}\lambda^{T}(z)>0$. The same arguments holds for the compact operator  $S$. 
 
Let $P(z)$ denote the orthogonal projection on $E_{T}(z)=E_{S}(z)$ and $Q(z)=I-P(z)$. We put $M(T)=\sup_{\lambda \in \mathbb{T}}\Vert K_{z}^{\rho}(T) \Vert$, for all $z\in \mathbb{T}$, we  also have 
  $$
  m(T)Q(z)\leq \tilde{K}_{z}^{\rho}(T)\leq  M(T) Q(z)  
  $$
  and
   $$
  m(S)Q(z)\leq \tilde{K}_{z}^{\rho}(S)\leq  M(s)Q(z).
  $$
  This two inequalities gives 
  $$ \frac{m(S)}{M(T)} \tilde{K}_{z}^{\rho}(T)\leq \tilde{K}_{z}^{\rho}(S)\leq \frac{M(S)}{ m(T)} \tilde{K}_{z}^{\rho}(T).$$
  Hence
   $$ \frac{m(S)}{M(T)} K_{z}^{\rho}(T)\leq K_{z}^{\rho}(S)\leq \frac{M(S)}{ m(T)} K_{z}^{\rho}(T),$$
	for all  $z\in \mathbb{T}$. Now, by the uniqueness of harmonic extension we also have 
   $$ \frac{m(S)}{M(T)} K_{z}^{\rho}(T)\leq K_{z}^{\rho}(S)\leq \frac{M(S)}{ m(T)} K_{z}^{\rho}(T),$$
	for all  $z\in \mathbb{D}$. This complete the proof of the theorem.
\end{proof} 
\begin{remark} In the previous theorem the hypothesis $\Gamma_p (T)$ is empty can be relaxed. In this case we can use the Corollary \ref{coro:he} and we  applied the Theorem \ref{thm:hpc1} for  $\tilde{T}$ and  $\tilde{S}$ as in the decomposition of $T$ and $S$ respectively, given by the Corollary \ref{coro:he}. 
\end{remark}
\begin{corollary}\label{cor:hpc12}Let $T \in C_{\rho}(H)\cap \mathcal{K}(H)$ with $w_{\rho}(T)=1$  and  $\Gamma_p (T)$ is empty. If there exits an unitary operator $U$ such that $U(ker (K_{z}^{\rho} (T)))\subseteq ker (K_{z}^{\rho} (T))$ for all $z\in \mathbb{T}$, then $U^*TU$ is Harnack equivalent to $T$.
\end{corollary}
\begin{proof}
We have $K_{z}^{\rho} (U^*TU)=U^*K_{z}^{\rho}U$ and $Ker (K_{z}^{\rho} (U^*TU))=Ker(K_{z}^{\rho}U)=U^*Ker(K_{z}^{\rho})$. As we see in the proof the preceding theorem that $K_{z}^{\rho}(T)=\rho I+ R_{z}^{\rho}(T)$ with $ R_{z}^{\rho}(T)$ is compact. Hence $Ker(K_{z}^{\rho})$ is finite dimensional, so the restriction of $U$ to $Ker(K_{z}^{\rho})$ is injective, equivalently to the restriction of $U^*$ to $Ker(K_{z}^{\rho})$ is surjective. Thus $U^*Ker(K_{z}^{\rho})=Ker(K_{z}^{\rho}(T))$ and  $Ker (K_{z}^{\rho} (U^*TU))= Ker(K_{z}^{\rho}(T))$  for all $z\in \mathbb{T}$. 
By  Theorem  \ref{thm:hpc1} we conclude that $T{\stackrel{H}{\sim}} S$. 
\end{proof}

\begin{corollary}\label{cor:hpc1} Let $T \in C_{1}(H)\cap \mathcal{K}(H)$ with $\left\|T\right\|=1$ and  $\Gamma_p (T)$ is empty. Then  $S \in C_{1}(H)\cap \mathcal{K}(H)$ is Harnack equivalent to $T$ if and only if $E=ker (I-T^*T)=ker (I-S^*S)$ and $T_{|E}=S_{|E}$.
\end{corollary}
\begin{proof} Let $T,S \in C_{1}(H)\cap \mathcal{K}(H)$ such that   $T{\stackrel{H}{\sim}} S$. 
By  Theorem  \ref{thm:hpc1}, $Ker K_{z} (T)=Ker K_{z} (S)$ for all $z\in \mathbb{T}$. 
On the other hand, the fact that $$K_{z} (T)= (I-zT^{*})^{-1}[  I- \left|z\right|^{2} T^{*}T] (I-\overline{z}T)^{-1},$$ we can easily deduce that 
$$kerK_{z} (T)=(I-\overline{z}T)(ker (I-  T^{*}T)) \quad \text{ for all } z \in \mathbb{T},$$
and similarly
 $$kerK_{z} (S)=(I-\overline{z}S)(ker (I-  S^{*}S)) \quad \text{ for all } z \in \mathbb{T},$$
Then we also  have,
$$(I- \overline{z}T)(ker (I-  T^{*}T))=(I- \overline{z}S)(ker (I-  S^{*}S))\quad \text{ for all } z \in \mathbb{T}.$$

We put  $E=ker (I-  T^{*}T)$. Let $x\in ker (I-  S^{*}S)$ and $z \in \mathbb{T}.$ Then
$(I-zS)x=(I-zT)y(z)$ with $y(z)\in E$, hence $y(z)=(I-zT)^{-1}(I-zS)x$ have an analytic extensions in a neighbourhood of $\overline{\mathbb{D}}$. It follows that $x=y(0)=\int_{0}^{2\pi}y(e^{i\theta})dm(\theta)\in E$, since $E$ is closed. This proves  $ker (I-  S^{*}S)\subseteq E$. Now the equality holds by interchanging the roles of $T$ and $S.$  Furthermore, we can see that for all $x\in E$, we have
$$y(z)=(I-zT)^{-1}(I-zS)x\in E   \quad \text{ for all } z \in \overline{\mathbb{D}}.$$
But
\begin{align*}
y(z) & =(I-zT)^{-1}(I-zT+z(T-S))x\\
&=x+\sum^{+\infty}_{n=1}z^{n} T^{n-1} (T-S)x
\end{align*}
On the other hand, we have
$$ T^{n-1} (T-S)x=\int_{0}^{2\pi}e^{-i n\theta}y(e^{i \theta})dm(\theta)\in E  \quad \text{ for all } n\geq 1,$$
and
$$\left\langle (I-  T^{*}T) T^{n-1}x ,T^{n-1} (T-S)x\right\rangle=0\quad \text{ for all } n\geq 1.$$
thus
$$\Vert (I-  T^{*}T) T^{n-1}(T-S)x \Vert^2 -\Vert T^{n-1} (T-S)x\Vert^2=0  \quad \text{ for all } n\geq 1,$$
so
$$\Vert (T-S)x \Vert =\Vert T^{n} (T-S)x\Vert^2   \longrightarrow 0,$$
because $r(T)<1$. This implies that $Tx=Sx$ for all $x\in E$.

Conversely, if $E=ker (I-T^*T)=ker (I-S^*S)$ and $T_{|E}=S_{|E}$, then for all  $z\in \mathbb{T}$, we have
$$Ker K_{z} (T)=(I- zT)(ker (I-  T^{*}T))=(I- zS)(ker (I-  S^{*}S))=Ker K_{z} (S).$$
Thus, by Theorem  \ref{thm:hpc1}, $T{\stackrel{H}{\sim}} S$.
\end{proof} 
For each $n\geq 1$, let 
$$  J_n =\begin{pmatrix}
0&1&0&\ldots  &0\\
0&0&\ddots&\ddots  &\vdots\\
\vdots&\vdots &\ddots&\ddots  & 0\\
0&0&\ldots&\ddots & 1\\
0&0&\ldots&\ldots & 0
\end{pmatrix}
 $$
 denote the (nilpotent) Jordan block of size $n$. By Corollary \ref{cor:hpc1} and the fact that $ker (I-J_{n}^{*} J_{n})=span\{e_2, \ldots, e_n\}$,   the Harnack parts of $J_n$ is given by
 \begin{corollary}The Harnack parts of $J_n$ is precisely the  set of all matrices of $C_{2}(\mathbb{C}^{n})$ of  the  form 
$$  M =\begin{pmatrix}
0&1&0&\ldots  &0\\
0&0&\ddots&\ddots  &\vdots\\
\vdots&\vdots &\ddots&\ddots  & 0\\
0&0&\ldots&\ddots & 1\\
z&0&\ldots&\ldots & 0
\end{pmatrix},
$$
where $z$ is in the open unit disc.
\end{corollary}
In the case of compact operators, we deduce from Theorem \ref{thm:numr} the next
result.
\begin{proposition}
Let $T_0, T_1 \in C_{\rho}(H)\cap \mathcal{K}(H)$ with $1 \leq \rho
\leq 2$, then we have:
\begin{enumerate}
    \item [(i)] Assume that $\rho=1$ and $T_1{\stackrel{H}{\prec}} T_0$, then $W(T_0)\cap \mathbb{T} =
	W(T_1)\cap \mathbb{T}$.
	\item [(ii)] Suppose that $1 < \rho
\leq 2$, $T_1{\stackrel{H}{\prec}} T_0$ and $\Gamma(T_0)=\emptyset$, then $W(T_0)\cap \mathbb{T} \subseteq 
	W(T_1)\cap \mathbb{T}$.
	\item [(iii)] If $T_1{\stackrel{H}{\sim}} T_0$, then
	$W(T_0)\cap \mathbb{T} = 
	W(T_1)\cap \mathbb{T}$.
\end{enumerate}
\end{proposition}
\begin{proof}By using the Theorem \ref{thm:numr} and Proposition \ref{rho1<<rho2}, it suffices to prove that $\overline{W(T)}\cap \mathbb{T}= W(T)\cap \mathbb{T}$ for each $T\in C_{2}(H)\cap \mathcal{K}(H)$. Indeed, let $\lambda \in \overline{W(T)}\cap \mathbb{T}$, then $\lambda$ is a limit of $\langle Tx_n , x_n\rangle$ for some sequence $(x_n)$ of unit vectors. Therefore, there exist a subsequence $(x_{j(n)})$ of $(x_n)$ such that $x_{j(n)}$ converge to some $x$ in the weak star topology. Since $T$ is a compact operator then $Tx_{j(n)} \longrightarrow Tx$  in the norm topology, this implies that  $\lambda=\langle Tx,x\rangle$, and hence $x \neq 0$. Consequently,  $\frac{\lambda}{\Vert x\Vert^2}\in W(T)\subseteq \overline{\mathbb{D}}$. So $\frac{1}{\Vert x\Vert^2}\leq 1$ and hence $\Vert x\Vert^2\geq 1$, but we also have $\Vert x\Vert^2\leq 1$, this means that $\Vert x\Vert = 1$ and $\lambda \in W(T)$.
\end{proof}
\subsection{Weak stability and Harnack domination}

One says that an operator is weakly stable if $\lim_{n\rightarrow+\infty}T^{n}  = 0$ in the weak topology of $B (H)$. Also we have that this is equivalent to $T^*$ is weakly stable.

We give the following proposition which is useful to study this property.
\begin{proposition}\label{inequalities-dilations}
Let $H$ be a separable Hilbert space. Then, we have

\begin{enumerate}

	\item [(i)] Let $T \in C_{\rho}(H)$ and denote by $V$ its minimal isometric $\rho$-dilation. Then, for every $m \geq 1$, we have 
$$
\Vert\sum_{k=1}^{m}V^{\ast k+1}x_k\Vert\leq \Vert\sum_{k=1}^{m}T^{\ast k}x_k\Vert\leq \rho\Vert \sum_{k=1}^{m}V^{\ast k}x_k \Vert
$$
for any m-tuple $\left( x_1,\cdots, x_m\right) $ of vectors of $H$.

     \item [(ii)]
Assume that $T_1$ be Harnack dominated by $T_0$ in $C_{\rho}(H)$ for a constant $c\geq 1$. If $V_i$ acting on $K_i \supseteq H$ is the minimal isometric $\rho$-dilation of $T_i$ $(i=0, 1)$, then we have
$$
\Vert\sum_{k=1}^{m}V_1^kx_k\Vert\leq c\Vert \sum_{k=1}^{m}V_0^kx_k \Vert
$$
for any m-tuple $\left( x_1,\cdots, x_m\right) $ of vectors of $H$.
\end{enumerate}
\end{proposition}
\begin{proof}
(i) Let $h=\sum_{i=0}^{n}V^{i}h_i$ with $h_i \in H$, then we have
\begin{align*}
\langle \sum_{k=1}^{m}T^{\ast k}x_k, Vh\rangle &= 
\sum_{k=1}^{m}\sum_{i=0}^{n}\langle T^{\ast k}x_k, V^{i+1}h_i \rangle
=\frac{1}{\rho}\sum_{k=1}^{m}\sum_{i=0}^{n}\langle T^{\ast k}x_k, T^{i+1}h_i \rangle\\
&=\sum_{k=1}^{m}\sum_{i=0}^{n}\langle V^{\ast k+i +1}x_k, h_i \rangle
=\langle \sum_{k=1}^{m}V^{\ast k+1}x_k, h\rangle
\end{align*}
Since the subset of all elements $h$ having the above form is dense in $K$, we get
\begin{align*}
 \Vert \sum_{k=1}^{m} T^{\ast k}x_k \Vert & 
 =\sup_{\Vert h \Vert=1}  \vert  \langle \sum_{k=1}^{m}T^{\ast k}x_k, h\rangle \vert 
\geq \sup_{\Vert h \Vert=1}  \vert  \langle \sum_{k=1}^{m}T^{\ast k}x_k, Vh\rangle \vert\\  
&\geq \sup_{\Vert h \Vert=1}  \vert  \langle \sum_{k=1}^{m}V^{\ast k+1}x_k, h\rangle
=\Vert \sum_{k=1}^{m} V^{\ast k+1}x_k \Vert
\end{align*}
and the left-hand side inequality is obtained. The right-hand side inequality is
obvious.\\

(ii) Now, suppose that $T_1\stackrel{H}{\underset{c}{\prec}}T_0$ in $C_{\rho}(H)$ and 
$V_i$ acting on $K_i \supseteq H$ is the minimal isometric $\rho$-dilation of $T_i$ $(i=0, 1)$. Using Theorem \ref{harnack}, we know that there exists an operator $S\in B(K_0 , K_1 )$ such that $S(H) \subset H$, $S|_{H}=I$, $SV_0 =V_{1}S$ and $\left\|S\right\|\leq c$. Let $\left( x_1,\cdots, x_m\right) $ be a $m$-tuple of vectors of $H$. Observe that $SV_0^k =V_{1}^kS$ for any positive integer $k$, thus we get
$$
\Vert\sum_{k=1}^{m}V_1^kx_k\Vert =\Vert\sum_{k=1}^{m}V_1^k Sx_k\Vert
=\Vert S\left[ \sum_{k=1}^{m}V_0^k x_k\right] \Vert
\leq c \Vert \sum_{k=1}^{m}V_0^k x_k \Vert.
$$
\end{proof}
\begin{lemma}\label{lem:WS} A $\rho$-contraction $T$ is weakly stable if and only if the minimal isometric $\rho$-dilation of $T$ is weakly stable.
\end{lemma}
\begin{proof}Let us assume that $T$ is weakly stable and $[V, K]$ is the minimal isometric $\rho$-dilation of  $T$. Hence  $T^*$ is also  weakly stable, i.e   $T^{*n}h \longrightarrow 0$ in the weak topology. Since $T^*$ has the Blum-Hanson property,  for each $h \in H$ and every increasing sequence $(k_n)_{n \geq 0}$ of positive integers, we have
$$\frac{1}{N}\sum_{n=0}^{N} T^{*k_{n}}h  \longrightarrow 0$$
in the norm topology. For each $N$, 
set $x_k=h/N$ if there exists an integer $n$ such that $k=k_n$ and $x_k=0$ otherwise, and use Proposition \ref{inequalities-dilations} (i). We derive
$$
\frac{1}{N} \sum_{n=1}^{N} V^{*k_{n}+1}h \longrightarrow 0.
$$
It is enough to ensure that 
\begin{equation}\label{Blum-Hanson_partial}
\frac{1}{N} \sum_{n=0}^{N} V^{*l_{n}}x \longrightarrow 0.
\end{equation} 
for any increasing sequence $(l_n)_{n \geq 0}$ of positive integers and any $x \in H$.
Now, let $x=\sum_{i=1}^{m} V^{i}x_{i}$ with $x_i \in H$, we easily deduce from 
\eqref{Blum-Hanson_partial} that
$$
\frac{1}{N} \sum_{n=1}^{N} V^{*l_{n}}x \longrightarrow 0.
$$ 
Since the subset of all elements $x$ having the above form is dense in $K$ and that the
sequence of operators $1/N\left[ \sum_{n=1}^{N} V^{*l_{n}}\right] $ is a sequence of contractions, we derive that $V^{\ast}$ has the Blum-Hanson property. Thus, the sequence $(V^{\ast n}x)$ weakly converge to $0$ for any $x \in K$. Hence $V$ is weakly stable.

Conversely, assume that $V$ is weakly stable. Then for each $(x,y) \in H^2$ and any $n\geq 1$, we have $\langle T^{n}x, y \rangle =\rho \langle V^{n}x, y \rangle \longrightarrow 0$. Hence, $T$ is weakly stable.
\end{proof}
\begin{corollary}\label{*stability}
Let $T_0$ and $T_1$ be two operators in $C_{\rho}(H)$. Then, we have:
\begin{enumerate}
	\item [(i)] Assume that $T_1$ be Harnack dominated by $T_0$ in $C_{\rho}(H)$ and that $T_0$ is weakly stable (resp. stable). Then $T_1$ is also weakly stable (resp. stable). 
	\item [(ii)] Let $T_0$ and $T_1$ be Harnack equivalent in $C_{\rho}(H)$. Then $T_0$ is weakly stable (resp. stable) if and only if $T_1$ is weakly stable (resp. stable).
		\end{enumerate}
\end{corollary}
\begin{proof}
(i) Assume that $T_0$ is weakly stable. Using Lemma \ref{lem:WS}, we see that the minimal isometric $\rho$-dilation $V_0$ is
weakly stable. Applying Proposition \ref{inequalities-dilations} (ii) and using the Blum-Hanson property as in the proof of Lemma \ref{lem:WS}, we deduce
than $V_1$ is weakly stable. Using again Lemma \ref{lem:WS}, we obtain the weak stability of $T_1$. 

Now, suppose that $T_0$ is stable. We deduce from Lemma 3.5 of \cite{CaSu} that
$V_0$ is stable. From Proposition \ref{inequalities-dilations} (ii) we derive that
$V_1$ is stable. Then, by Lemma 3.5 of \cite{CaSu} we obtain the stability of $T_1$.

The assertion (ii) is a direct consequence of (i).  
\end{proof}
\begin{remark}
1) Concerning the stability of two Harnack equivalent $\rho$-contractions, the assertion (ii) is exactly Corollary 3.6 of \cite{CaSu}. 

2) Since any $\rho$-contraction $T$ is similar to a contraction and power bounded, by \cite[Proposition 8.5]{Kub},  the residual spectrum $\sigma_{r}(T) $ of $T$ is included in  $\mathbb{D}$. By \cite[Proposition 8.4]{Kub} it follows that if any $\rho$-contraction $T$ is weakly stable then $\sigma_{p}(T)\subseteq \mathbb{D}$. In this case, according to Lemma \ref{lem:WS}, if $V$ is  the minimal isometric $\rho$-dilation of $T$, then  $\Gamma (V)=\sigma_{c}(V)$. So, if there exist $\lambda \in \sigma_{p}(T)$ such that $\vert \lambda \vert=1$ then $T$ is not weakly stable and this $\rho$-contraction  is in Harnack part of an operator  with $\rho$-\emph{numerical radius} one.
\end{remark}
\section{Examples of Harnack parts for some nilpotent matrices  with numerical radius one }\label{sec:3}

In the following, we try to describe the Harnack parts a nilpotent matrices  with numerical radius one. We begin by the nilpotent matrix of order one in the dimension two.  
\begin{theorem}\label{thm02} Let $ T_0 =
 \begin{pmatrix}
0 & 2 \\
0 & 0
\end{pmatrix} \in C_{2}(\mathbb{C}^{2})$, then the Harnack parts of $T_0$ is reduced to $\{T_0\}$.
\end{theorem}
\begin{proof} Let $ T\in C_{2}(\mathbb{C}^{2})$ such that $T{\stackrel{H}{\sim}} T_0$, then by Theorem \ref{harnack}, there exists $\alpha, \beta>0$ $(\alpha\leq 1$, $\beta \geq 1$) such that 
\begin{equation}\label{equ22}\alpha K_{z}^{2} (T_0) \leq  K_{z}^{2} (T)\leq \beta K_{z}^{2} (T_0), \quad \text{ for all } z\in \mathbb{D}.
\end{equation}
By Corollary \ref{harnackequ}, the operator $T$ not admits eigenvalues in $\mathbb{T}$. Hence, $K_{z}^{2} (T_0)$ and  $ K_{z}^{2}(T)$ are uniformly bounded in $\mathbb{D}$ and may be extended to a positive operators on  $\overline{\mathbb{D}}$.

We have $$K_{z}^{2}(T_0)= 2\begin{pmatrix}
1 & \overline{z} \\
z & 1
\end{pmatrix},$$
thus $\det (K_{z}^{2}(T_0))=4(1-\left|z\right|^{2})$ and $Ker(K_{z}^{2}(T_0))$ of rang one on $\mathbb{T}$. Let $v(z)= \begin{pmatrix}
1 \\
-z
\end{pmatrix}$, then $K_{z}^{2}(T_0)v(z)=0$ on $\mathbb{T}$. This implies by \eqref{equ22} that

\begin{equation}\label{v(z)}
0=K_{1,\theta}^{2}(T)v(e^{i\theta})=  K_{1,\theta}^{2}(T)e_1 -e^{i\theta} K_{1,\theta}^{2}(T) e_2 =0  \quad \text{ for all } \theta \in \mathbb{R}.
\end{equation} 
Multiplying successively \eqref{v(z)} by $1$ and $e^{-i\theta}$, and integrating with respect the Haar measure $m$ on the torus, we obtain: $Te_2 =2 e_1$ and $T^{*}e_1 =2 e_2$. Thus $T$ take the form 
 $$T= \begin{pmatrix}
0 & 2 \\
b & 0
\end{pmatrix},$$
with $b\in \mathbb{C}$. Since $w(T)\leq 1$, we have
$$ \left|2 x_2 \overline{x_1} + b x_1\overline{x_2} \right| \leq \left|x_1\right|^{2} + \left|x_2\right|^{2}.$$
If we take $x_1 =\frac{\sqrt{2}}{2}$ and $x_2 =\frac{\sqrt{2}}{2}e^{i\theta},$ we get
$$
\left|1 + b e^{-2i\theta}\right| \leq 1.
$$
In particular, for  $\theta =\frac{\arg b}{2}$
$$
1 +\left| b \right| \leq 1
$$ 
This implies that $b=0$ and $T=T_0$.
\end{proof}
In the following result, we describe the Harnack parts of  a nilpotent matrix of order two in $C_{2}(\mathbb{C}^{n})$, $n\geq 3$,    with numerical radius one.
\begin{theorem}\label{thm: nilp2}Let  $N\in C_{2}(\mathbb{C}^{n})$, $n\geq 3$ such that   $$  N =\begin{pmatrix}
0&0&\ldots& &a\\
0&0&\ldots& &0\\
\vdots&\vdots&\ddots& &\vdots \\
0&0&\ldots& & 0\\
\end{pmatrix}
 $$
with $\left|a\right|=2$, then the Harnack parts of $N$ is the  set of all matrices of $C_{2}(\mathbb{C}^{n})$ of  the  form 
\begin{equation}\label{equ;harnakn}
T =\begin{pmatrix}
0&0&a\\
0& B &0\\

0&0& 0
\end{pmatrix}
\end{equation}
   with $B\in C_{2}(\mathbb{C}^{n-2})$ such that $w(B)<1$.
\end{theorem}
\begin{proof}Let $ T\in C_{2}(\mathbb{C}^{n})$ such that $T{\stackrel{H}{\sim}} N$. By Corollary \ref{harnackequ}, the operator $T$ not admits eigenvalues in $\mathbb{T}$. Hence, $K_{z}^{2} (N)$ and  $ K_{z}^{2}(T)$ are uniformly bounded in $\mathbb{D}$ and may be extended to a positive operators on  $\overline{\mathbb{D}}$. We have
$$K_{z}^{2}(N)= \begin{pmatrix}
2 & 0 & a  \overline{z}\\
0 & 2I_{n-2} & 0\\
\overline{a} z & 0 & 2
\end{pmatrix},$$
with $I_{n-2}$ denote the identity matrix on the linear space spanned by the vectors  $e_2, \ldots, e_{n-1}$ of the canonical basis of $\mathbb{C}^{n}$.
Then   $\det (K_{z}^{2}(N))=2^{n-2}(4-\left|a\right|^{2}\left|z\right|^{2})$. Let 
$v(z)=-a\overline{z}e_1+2e_n$, then $K_{z}^{2}(N)v(z)=0$ on $\mathbb{T}$. Thus by proposition \ref{prop2.3}, $K_{z}^{2}(T)v(z)=0$ on $\mathbb{T}$. This implies that
\begin{equation}\label{equ:niln}
-ae^{-i\theta} K_{1,\theta}^{2}(T)e_1 +2 K_{1,\theta}^{2}(T) e_n =0  \quad \text{ for all } \theta \in \mathbb{R}.
\end{equation}
Multiplying successively \eqref{equ:niln} by $1$, $e^{i\theta}$, $e^{-i\theta}$ and $e^{2i\theta}$, and integrating with respect $m$, we obtain:    
\begin{equation}\label{equ:nil1n}
T^{*}e_1 =\overline{a} e_n, Te_n =a e_1,
T^{*} e_n =0 \text{ and } Te_1  = 0.
\end{equation} 
By  \eqref{equ:nil1n} , the matrix  $T$ take the form \eqref{equ;harnakn}. Hence
\begin{equation}\label{Kz(T)}
K_{z}^{2}(T)= \begin{pmatrix}
2 & 0 & a  \overline{z}\\
0 & K_{z}^{2}(B)& 0\\
\overline{a} z & 0 & 2
\end{pmatrix}
\end{equation} 
By Theorem \ref{thm:hpc1}, we know that $ker(K_{z}^{2}(T))=ker(K_{z}^{2}(N))$ for all 
$ z \in \mathbb{T}$, it forces $ker(K_{z}^{2}(B))$ to be equal to $\left\lbrace  0 \right\rbrace $ for every $ z \in \mathbb{T}$.
Using again Theorem \ref{thm:hpc1}, we deduce that $B$ is Harnack equivalent to $0$, thus $w(B)<1$.

Conversely, Let $T\in C_{2}(\mathbb{C}^{n})$ given  by \eqref{equ;harnakn}, then
we can write $K_{z}^{2}(T)$ under the form given by \eqref{Kz(T)}.
Since $B\in C_{2}(\mathbb{C}^{n-2})$ with $w(B)<1$, $B$ is Harnack equivalent to  $0$ in $C_{2}(\mathbb{C}^{n-2})$. Then  by Theorem \ref{harnack}, there exists $\alpha, \beta>0$ $(\alpha\leq 1$, $\beta \geq 1$) such that 
\begin{equation*}2\alpha I_{n-2}  \leq  K_{z}^{2} (B)\leq 2\beta I_{n-2} , \quad \text{ for all } z\in \mathbb{D}.
\end{equation*}
Thus
\begin{equation*}\alpha  K_{z}^{2} (N) \leq  K_{z}^{2} (T)\leq \beta K_{z}^{2} (N), \quad \text{ for all } z\in \mathbb{D}.
\end{equation*}
This means that $T$ is Harnack equivalent to $N$.
\end{proof}
\begin{theorem}Let $ N=
 \begin{pmatrix}
0 & a & 0 \\
0 & 0 & a\\
0 & 0 &  0
\end{pmatrix}$ such that  $\left|a\right|=\sqrt{2}$, then the Harnack parts of $N$ is the  set of all matrices of $C_{2}(\mathbb{C}^{3})$ of  the  form $$ T=a
 \begin{pmatrix}
0 & e^{-i\theta} & 0 \\
0 & 0 & e^{i\theta}\\
0 & 0 &  0
\end{pmatrix}, \quad \theta \in \mathbb{R}.$$ 
\end{theorem}
\begin{proof}Let $ T\in C_{2}(\mathbb{C}^{3})$ such that $T{\stackrel{H}{\sim}} N$. By Corollary \ref{harnackequ}, the operator $T$ not admits eigenvalues in $\mathbb{T}$. Hence, $K_{z}^{2} (N)$ and  $ K_{z}^{2}(T)$ are uniformly bounded in $\mathbb{D}$ and may be extended to a positive operators on  $\overline{\mathbb{D}}$. Furthermore, by \cite[Theorem 5.2]{Cassier_2} $T^2 {\stackrel{H}{\sim}} N^2$, then by Theorem \ref{thm: nilp2}, the operator $T^2$ takes the following form
$$ T^2=
 \begin{pmatrix}
0 & 0 & a^2 \\
0 & b & 0\\
0 & 0 &  0
\end{pmatrix},$$ with $\left|b\right|<1$. If $b\neq 0$ then $Ker T^2=\mathbb{C}e_1$ is invariant by $T$, so $Te_1=xe_1$ but $0=T^2 e_1=x^2e_1$, this implies that $x=0$ and $Te_1=0$. Similarly, $ \mathbb{C}^3\neq Im T\supseteq Im T^2=span\{e_1, e_2\} $ is invariant by $T$, so $Te_2=ue_1+ve_2$ for some $u,v\in \mathbb{C}$. On the other hand, we have
$$K_{z}^{2}(N)= \begin{pmatrix}
2 &a  \overline{z} & a^2  \overline{z}^2\\
\overline{a} z & 2 & a  \overline{z}\\
\overline{a}^2 z^2 & \overline{a} z & 2
\end{pmatrix},$$
thus $\det (K_{z}^{2}(N))=4(2-\left|a\right|^{2}\left|z\right|^{2})$, so  $Ker(K_{z}^{2}(N))$ of rank one on $\mathbb{T}$. Let $v(z)=- a^2  \overline{z}e_1+2ze_2 $, then $K_{z}^{2}(N)v(z)=0$ on $\mathbb{T}$. Thus by Proposition \ref{prop2.3}, $K_{z}^{2}(T)v(z)=0$ on $\mathbb{T}$. This implies that
\begin{equation}\label{equ:nilpo3}
-a^2 e^{-i\theta} K_{1,\theta}^{2}(T)e_1 +2 e^{i\theta}  K_{1,\theta}^{2}(T) e_3 =0  \quad \text{ for all } \theta \in \mathbb{R}.
\end{equation}
Using \eqref{equ:nilpo3} in a similar way than before, we get  
\begin{equation}\label{equ:nilpo33}
 2T e_3 =a^2 T^{*}e_1 \quad \text{ and }\quad 2T^{*} e_3 =a^2 Te_1.
 \end{equation} 
By this we deduce that 
$$ \langle T e_3,e_1 \rangle=\frac{•a^2}{2}  \langle T^{*}e_1 ,e_1 \rangle= \frac{•a^2}{2} \langle e_1 ,Te_1 \rangle =0$$
and
$$ \langle T e_3,e_3 \rangle=\frac{•a^2}{2}  \langle T^{*}e_1 ,e_3 \rangle= \frac{•a^2}{2} \langle e_1 ,Te_3 \rangle =0.$$
 The matrix  $T$ take the form
 $$ T=
 \begin{pmatrix}
0 & u & 0 \\
0 & v & w\\
0 & 0 &  0
\end{pmatrix}.$$

By \eqref{equ:nilpo33}, $2w=a^2 T^*e_1=a^2\overline{u}e_2$, hence 
\begin{equation}\label{equ:nilpo333}
\overline{a}w=a\overline{u}.
 \end{equation} 
This implies that $u$ and $v$ must be not equal to $0$. Now the fact that
 $$ T^2=
 \begin{pmatrix}
0 & uv & uw \\
0 & v^2 & wv\\
0 & 0 &  0
\end{pmatrix}=\begin{pmatrix}
0 & 0 & a^2 \\
0 & b & 0\\
0 & 0 &  0 \end{pmatrix}.$$
implies that $v=b=0$ and

\begin{equation}\label{equ:nilpo3334}
uw=a^2.
\end{equation} 
 By \eqref{equ:nilpo333} and \eqref{equ:nilpo3334} we can deduce that $u=ae^{-i\theta}$  and $w=ae^{i\theta}$, $\theta \in \mathbb{R}$.

Conversely, Let $T\in C_{2}(\mathbb{C}^{3})$ given as  above, then
$$K_{z}^{2}(T)= \begin{pmatrix}
2 & u\overline{z}& a^2  \overline{z}^2\\
\overline{u} z & 2 & w\overline{z}\\
\overline{a}^2 z^2 & \overline{w} z & 2
\end{pmatrix}$$
By  a simple calculus, we can see that 
$$det( \beta K_{z}^{2}(N)-K_{z}^{2}(T))\geq \beta^3+ q(\beta)$$
with $q$ is a polynomial with constant coefficient of degree two on $\mathbb{T}$. so for $\beta$ sufficiently large, we assert that there exists a constant $r>0$ such that 
$$det( \beta K_{z}^{2}(N)-K_{z}^{2}(T))\geq \beta^3+ q(\beta)\geq r>0.$$
This is exactly the  principal minor of order $3$ of $\beta K_{z}^{2}(N)-K_{z}^{2}(T)$. Similarly, we can prove that the  principal minor of order $2$ is also positive for $\beta$ sufficiently large. By the Sylvester's criterion for the  positive-semidefinite Hermitian matrices, we assert that
$$ \beta K_{z}^{2}(N)-K_{z}^{2}(T))\geq 0\quad \quad  \text{ for all } z\in \overline{\mathbb{D}}.$$
We prove similarly the existence  of $0<\alpha \leq 1$ such that
$$K_{z}^{2}(T)-\alpha  K_{z}^{2}(N) \geq 0 \quad \quad  \text{ for all } z\in \overline{\mathbb{D}}.$$
\end{proof}
\begin{remark}We can see that the matrix $T$ given in the above theorem takes the form  
$$T=U_{\theta}^{*} N U_\theta \quad \text{ with } \quad U_\theta =
 \begin{pmatrix}
e^{i\theta} & 0& 0\\
0 & 1&0\\
0 & 0 &  e^{i\theta} 
\end{pmatrix},\qquad \theta \in \mathbb{R}.$$
Thus all the elements of the Harnack parts of $N$ are unitary equivalent to $N$. 
\end{remark}

\bibliographystyle{amsplain}

\end{document}